\documentclass{article}

\usepackage{arxiv}

\usepackage[utf8]{inputenc} 
\usepackage[T1]{fontenc}    
\usepackage{hyperref}       
\usepackage{url}            
\usepackage{booktabs}       
\usepackage{amsfonts}       
\usepackage{nicefrac}       
\usepackage{microtype}      
\usepackage{lipsum}		
\usepackage{graphicx}

\usepackage{doi}

\usepackage{float}

\usepackage{tikz}

\usetikzlibrary{decorations.markings,arrows}

\usepackage{graphicx}
\usepackage{dcolumn}
\usepackage{bm}

\usepackage[english]{babel}
\usepackage{amsmath}
\usepackage{amssymb}
\usepackage{amsthm}
\usepackage[version=3]{mhchem}
\usepackage{ mathrsfs }
\usepackage{dirtytalk}

\theoremstyle{plain}

\newtheorem{theorem}{Theorem}[section]
\newtheorem{corollary}{Corollary}[theorem]
\newtheorem{lemma}[theorem]{Lemma}
\newtheorem{algorithm}{Algorithm}[section]

\theoremstyle{definition}
\newtheorem{definition}[theorem]{Definition}

\newtheorem{innercustomgeneric}{\customgenericname}
\providecommand{\customgenericname}{}
\newcommand{\newcustomtheorem}[2]{%
  \newenvironment{#1}[1]
  {%
   \renewcommand\customgenericname{#2}%
   \renewcommand\theinnercustomgeneric{##1}%
   \innercustomgeneric
  }
  {\endinnercustomgeneric}
}

\newcustomtheorem{customthm}{Theorem}
\newcustomtheorem{customlemma}{Lemma}

\newtheorem*{theorem*}{Theorem}

\newcommand{\R}{\mathbb{R}}

\newcommand{\Z}{\mathbb{Z}}

\newcommand{\N}{\mathbb{N}}

\newcommand{\proj}{\text{proj}}
\newcommand{\inv}[1]{{#1}^{-1}}

\newcommand{\abs}[1]{\left\lvert#1\right\rvert}

\newcommand{\dimh}{\textrm{dim}_{H}}
\newcommand{\dimm}{\textrm{dim}_{M}}
\newcommand{\dimup}{\textrm{dim}_{M^*}}

\newcommand{\nstdlm}[1]{{\underset{#1}{\textrm{lim}}}}

\usepackage{siunitx}

\title{The Fractal Dimension of Product Sets}

\author{{Machiel van Frankenhuijsen}\thanks{http://research.uvu.edu/machiel/index.html} \\
	Department of Mathematics\\
	Utah Valley University\\
	Orem, UT  84058 \\
	\texttt{vanframa@.edu} \\
	\And
	\href{https://orcid.org/0000-0002-3764-6361}{\includegraphics[scale=0.06]{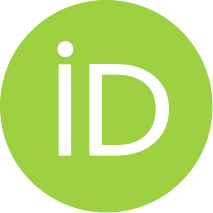}\hspace{1mm}Clayton Williams} \\
	Department of Mathematics\\
	Brigham Young University\\
	Provo, UT   84606\\
	\texttt{clayton.williams@mathematics.byu.edu} \\
}

\renewcommand{\shorttitle}{\textit{arXiv} Template}

\hypersetup{
pdftitle={The Fractal Dimension of Product Sets},
pdfsubject={math.GN
},
pdfauthor={Machiel ~van Frankenhuijsen, Clayton ~Williams},
pdfkeywords={fractal geometry, Minkowski dimension, ultralimits},
}

\begin{document}
\maketitle

\begin{abstract}
Using ultraproduct techniques we define a nonstandard Minkowski dimension which exists for all bounded sets and which has the property that $\dim(A\times B)=\dim(A)+\dim(B).$ That is, our new dimension is product-summable. To illustrate our theorem we generalize an example of Falconer's to show that the standard upper Minkowski dimension, as well as the Hausdorff dimension, are not product-summable.
\end{abstract}


\section{Introduction}

There are several types of dimension used in fractal geometry, which coincide for some sets but have important, distinct properties. Indeed, determining the most proper notion of dimension has been a major problem in geometric measure theory from its inception, and debate over what it means for a set to be fractal has often reduced to debate over the proper notion of dimension. A classical example is the \say{Devil's Staircase}, which is intuitively fractal but which has integer Hausdorff dimension \cite[page 82]{mandelbrotnature},\cite[page 337]{lapvanf2013}. As Falconer notes in \textit{The Geometry of Fractal Sets}, \cite[pages ix-x]{falconer1986geometry}, the Hausdorff dimension is \say{undoubtedly, the most widely investigated and most widely used} type of dimension. It can, however, be difficult to compute or even bound (from below) for many sets. For this reason one might want to work with the Minkowski dimension, for which one can often obtain explicit formulas. Moreover, the Minkowski dimension has many useful identities. In particular, the Minkowski dimension of a product set is the sum of the dimensions of the factoring sets. We define this idea below for clarity.\par 

\begin{definition}
Let $f$ be a function defined for all bounded subsets of $\R^k$ for all $k.$ Then $f$ is \textit{product-summable} if $f(A)+f(B)=f(A\times B),$ where $A\times B\subset\R^{m+n}$ is the Cartesian product of $A$ and $B,$ for all such bounded subsets $A\subset \R^m$ and $B\subset\R^n$.
\end{definition}

Product-summability is a desirable property for any dimension, by analogy with boxes in $\R^n$. The Hausdorff dimension is not product-summable, as shown in \cite{marstrand1954prod}, nor is the upper Minkowski dimension. The Minkowski dimension is product-summable, as shown in \cite[corollary 2.5]{robinson2013strict}, but is unfortunately defined in terms of a limit which does not exist for all sets. In this paper we introduce a nonstandard extension of the Minkowski dimension which exists for all bounded sets, agreeing with the standard Minkowski dimension whenever it exists. Moreover, this nonstandard Minkowski dimension is product summable. We make the argument that the nonstandard Minkowski dimension is an appropriate tool of analysis for certain problems in geometric measure theory.\par

This is not the first treatment of product problems in fractal geometry. Freilich and Marstrand both discussed this problem in the context of the Hausdorff dimension, more recently Robinson and Sharples construct pathological products for the upper and lower Minkowski dimensions \cite{freilich1950measure},\cite{marstrand1954prod},\cite[theorem 5.11]{falconer1986geometry},\cite{robinson2013strict}. Our main result, theorem \ref{nonstdmink product summable}, generalizes the Minkowski dimension to one which is product-summable, exists for all bounded sets, and agrees with the standard Minkwoski dimension when it exists.\par 

Nonstandard analysis has been applied to geometric measure theory by previous authors. For example, Potgieter extends the Hausdorff measure to the nonstandard universe (in so doing introducing a nonstandard Hausdorff dimension) \cite{potgieter}. Potgieter introduces the nonstandard Hausdorff dimension for reasons differing from our own --- he does so because it allows for more intuitive proofs of some classical results. In our case we use an ultralimit to extend the Minkowski dimension to sets which do not have a standard Minkowski dimension. Because the Hausdorff dimension exists for every set our motivation does not apply to nonstandard extensions of the Hausdorff dimension. Potgieter's paper, however, shares a similarity with our own in that it gives a nonstandard definition of the Hausdorff dimension agreeing with the standard dimension \cite[theorems 4.3, 4.4]{potgieter}, while we define a nonstandard Minkowski dimension agreeing with the standard one when it exists.\par

Because this subject is of interest to both geometers and logicians we define most of the elementary ideas for clarity. In particular, we introduce ultralimits in an intuitive and conceptually transparent way based on Terence Tao's pedagogical discussion in \textit{Ultrafilters, Nonstandard Analysis, and Epsilon Management} \cite{tao_2007}. In our work we need only ultralimits and do not construct a nonstandard model.\par

We begin by proving the product-summability of the standard Minkowski dimension. We then introduce ultrafilters and nonstandard limits in section \ref{ultrasection}, developing only the tools we will use, and define the nonstandard Minkowski dimension. Our main result, theorem \ref{nonstdmink product summable}, establishes the product-summability of the nonstandard Minkowski dimension from its definition and a covering lemma, lemma \ref{lemma s atimesb}.  In section \ref{sets not product summable section} we provide an interesting generalization of an example of Falconer's, \cite[page 73]{falconer1986geometry}, for which the Hausdorff dimension of a product is not the sum of the dimensions of the factoring sets. The standard Minkowski dimension does not exist for these sets, which also provide an example for when the upper Minkowski and Hausdorff dimensions are not product-summable.

\section{The Standard Minkowski Dimension of a Product}

Recall that a set $A$ is totally bounded if for every $\varepsilon>0$ there exists a finite cover of $A$ by $\varepsilon$-balls, and that totally bounded sets are equivalent to bounded sets in Euclidean space \cite[section 1.1]{bishop2017fractals}. To proceed we define the  Minkowski dimension, which is computed using sets of fixed size. We introduce the Minkowski dimension via the covering number $N_A(\varepsilon).$ 

\begin{definition}\label{coveringnum}
Let $A$ be a bounded subset of $\R^n$ and $\varepsilon>0.$ Then the \textit{covering number} $N_A(\varepsilon)$ is the minimum number of cubes of side length $\varepsilon$ in $\R^n$ needed to cover $A.$
\end{definition}

Note that the Minkowski dimension can be computed using balls or cubes \cite{bishop2017fractals}. We can now define the standard Minkowski dimension.
\begin{definition}\label{mink}
Let $A\subset\R^n$ be bounded, $\varepsilon>0,$ and $N_A(\varepsilon)$ be its covering number. The \textit{Minkowski} dimension of $A,$ $\dimm(A)$, if the limit exists, is
\begin{align}
    \dimm(A) &= \lim\limits_{\varepsilon\to 0^+}\frac{\log N_A(\varepsilon)}{\log\frac{1}{\varepsilon}}.\label{minkup}
\end{align}
Define the upper Minkowski dimension similarly as $\dimup(A) = \limsup\limits_{\varepsilon\to 0^+}\frac{\log N_A(\varepsilon)}{\log\frac{1}{\varepsilon}}$; this exists for all bounded sets and equals the standard Minkowski dimension when it exists.
\end{definition}

In this section we show the standard Minkowski dimension of a Cartesian product is equal to the sum of the Minkowski dimensions of the factors, when the dimensions exist. Hence the Minkowski dimension is product-summable.\par

While coverings by balls are conceptually simple, coverings by regular $n$-cubes are often more convenient computationally (particularly when considering products, since the product of cubes is a higher dimensional cube). Indeed, we can go further and restrict our attention to dyadic cubes of fixed side length and position. Doing so gives us great control over the cover without sacrificing anything in terms of the behavior of the fractal sets we are able to observe, introducing only a constant into our computations.\footnote{The cover by dyadic cubes of fixed size used here, $\mathcal{N}(2^{-k}),$ is a subset of the more general cover Falconer uses to prove the Hausdorff dimension is not product-summable. Indeed, using general dyadic cubes is equivalent to using arbitrary covers for the purpose of computing the Hausdorff dimension. See Falconer \cite[chapter 5]{falconer1986geometry}.}\par

\begin{lemma}\label{doubling lemma}
Let $A\subset\R^n$ be bounded and $\varepsilon>0.$ Then
\begin{align*}
    N_A(\varepsilon)\leq 2^nN_A(2\varepsilon).
\end{align*}
\end{lemma}
\begin{proof}
It takes exactly $2$ intervals in $\R$ of length $\varepsilon$ to cover an interval of length $2\varepsilon,$ hence it takes no more than $2^n$ cubes in $\R^n$ of side length $\varepsilon$ to cover a subset of a cube of side length $2\varepsilon.$ Then a minimal cover of $A$ by cubes of side length $\varepsilon$ contains no more than $2^nN_A(2\varepsilon)$ cubes. 
\end{proof}
Using the above doubling lemma we can show the Minkowski dimension can be computed using dyadic cubes. Let $\mathscr{N}(2^{-k})$ be the set of dyadic cubes of side length $2^{-k}$ in $\R^k,$ that is, $\mathscr{N}(\varepsilon) = \{[m_12^{-k},(m_1+1)2^{-k})\times...\times[m_n2^{-k},(m_n+1)2^{-k}):m_j\in \Z\}.$ Define $S_A(2^{-k})$ to be the minimal number of cubes in $\mathscr{N}(2^{-k})$ covering $A.$

\begin{lemma}\label{dyadic cubes lemma}
Let $A\subset\R^n$ be bounded.  Then $
   2^{-n} N_A(2^{-k})\leq
    S_A(2^{-k})\leq 2^nN_A(2^{-k})$.
\end{lemma}

\begin{proof}
For $\varepsilon>0$ let $k\in\Z$ be such that $2^{-(k+1)}<\varepsilon\leq 2^{-k}.$ We have $N_A(\varepsilon)\leq 2^nN_A(2\varepsilon),$ similarly $S_A(2^{-(k+1)})\leq 2^nS_A(2^{-k}).$ Since $2^{-(k+1)}<\varepsilon\leq 2^{-k}$ and $S_A(2^{-k})$ is computed over a more restrictive type of cover than $N_A(\varepsilon),$ we have $N_A(2^{-k})\leq N_A(\varepsilon)\leq N_A(2^{-(k+1)})\leq S_A(2^{-(k+1)}).$ Then $2^{-n}N_A(\varepsilon)\leq N_A(2\varepsilon)\leq S_A(2^{-k}).$\par

In the other direction we have $S_A(2^{-k})\leq 2^{n}N_A(\varepsilon)$, since an interval in $\R$ of length $\varepsilon$ can be covered with two intervals of length $2^{-k}$. 
\end{proof}

We can now show the standard Minkowski dimension can be computed using covers by dyadic cubes. Dyadic cubes have the particularly useful net property, in which two dyadic cubes are either disjoint or one is a subset of the other. Because we're restricting our attention to cubes of constant size the dyadic cubes in $\mathscr{N}(\varepsilon)$ are either disjoint or equal, forming a regular tiling of $\R^n.$ Other approaches to geometric measure theory, such as the Hausdorff measure, take the infimum over all covers with diameter bounded above by $\varepsilon$. The Hausdorff $s$-dimensional outer measure can also be approximated by dyadic cubes, the so-called net measure, meaning the Hausdorff dimension can be computed using covers by dyadic cubes \cite[pages 64-65]{falconer1986geometry}.\par

\begin{lemma}\label{compute std mink by dyadic cubes lemma}
Let $A\subset\R^n$ be bounded, $\varepsilon>0,$ and $k\in \Z$ be such that $2^{-(k+1)}<\varepsilon\leq 2^{-k}.$ Let $S_A(2^{-k})$ be as above. Then if $\dimm(A)$ exists we have
\begin{align*}
    \dimm(A)=\lim_{\varepsilon\to0^+}\frac{\log S_{A}(2^{-k})}{\log\frac{1}{\varepsilon}}.
\end{align*}
\end{lemma}
\begin{proof}
By lemma \ref{dyadic cubes lemma} we have $\log\left( 2^{-n}N_A(2^{-k})\right) \leq \log\left( S_A(2^{-k})\right)\leq \log\left( 2^{n}N_A(2^{-k})\right).$ Since $\lim_{\varepsilon\to0^+}\frac{\log 2^n}{\log\frac{1}{\varepsilon}}=0$ the result follows.
\end{proof}

The net property of dyadic cubes allows us to prove that when the standard Minkowski dimension exists it is product summable. To do so we will need a lemma on covers of a product set.

\begin{lemma}\label{lemma s atimesb}
For bounded sets $A\subset\R^m$ and $B\subset\R^n$ $S_{A\times B}(2^{-k})=S_A(2^{-k})\times S_B(2^{-k}).$
\end{lemma}

\begin{proof}
Let $\{U_i\}$ be a cover of $A$ by $S_A(2^{-k})$ dyadic cubes of side length $2^{-k}$. Let $f_a$ be the fiber in $A\times B\subset\R^{m+n}$ of $a$ in $A$. For each fiber $f_a$ a minimum of $S_B(2^{-k})$ dyadic cubes of side length $2^{-k}$ is necessary to cover $f_a$. For each $U_i$ pick an $a_i\in U_i\cap A$, note $a_i\not\in U_j$ if $i\neq j.$ Hence there are $S_A(2^{-k})$ such points $a_i.$ Covering all $f_{a_i}$ by $S_B(2^{-k})$ dyadic cubes requires then $S_A(2^{-k})S_B(2^{-k})$ dyadic cubes. Since these cubes are disjoint no more efficient cover is possible at this size, hence $S_{A\times B}(2^{-k})= S_A(2^{-k})S_B(2^{-k}).$
\end{proof}

Lemmas \ref{compute std mink by dyadic cubes lemma} and \ref{lemma s atimesb} give us the product-summability of the standard Minkowski dimension.

\begin{corollary}\label{minkdimprodeq}
Let $A\subset\R^m,B\subset\R^n$ be bounded sets with Minkowski dimensions $\dimm(A)$ and $\dimm(B)$ respectively. Then $\dimm(A\times B) = \dimm(A)+\dimm(B).$
\end{corollary}

\section{The Non-standard Minkowski Dimension}\label{ultrasection}
In this section we will show that it is possible to extend the standard Minkowski dimension so it exists for all bounded subsets of $\R^n$ while retaining the product-summability property. Before doing so it is appropriate to remark that the upper Minkowski dimension is not such an extension, as will be shown in section \ref{sets not product summable section}. In order to properly extend the standard Minkowski dimension we will use some ultraproduct techniques.

\subsection{Filters and Ultrafilters}
A proper filter on an infinite set is a set of so-called \say{large} subsets, in a way made precise below.
\begin{definition}
A proper filter $\mathcal{Q}$ on a set $X$ is a proper subset of the power set of $X,$ $P(X)$, such that
\begin{enumerate}
    \item $X\in\mathcal{Q}$,
    \item If $A,B\in\mathcal{Q}$ then $A\cap B\in\mathcal{Q},$
    \item If $A\in\mathcal{Q}$ and $B\supseteq A$ then $B\in \mathcal{Q}.$
\end{enumerate}

\end{definition}

An ultrafilter is a proper filter which is maximal, that is, every subset of $P(X)$ is either \say{large} or \say{small}. An example of such a filter can be given as follows: let $a\in X$, then the set $\mathcal{Q}_a = \{A\subset X: a\in A\}$ has the property that for any subset $B$ of $X$, either $B\in \mathcal{Q}_a$ or $X\setminus B\in \mathcal{Q}_a.$ This type of filter is called a principal ultrafilter. We are interested in excluding this class of filter, in particular, the types of ultrafilters we are interested in also have the additional restriction that they contain all co-finite sets.  

\begin{definition}\label{defintion nonprinc uf}
A \textit{non-principal ultrafilter} $\mathcal{Q}$ on a set $X$ is a filter with the following properties:\cite[pages 204-205]{hrbacek1984introduction}
\begin{enumerate}
\setcounter{enumi}{3}
    \item For any subset $A$ of $X$, $A\in \mathcal{Q}$ or $X\setminus A\in \mathcal{Q}$,
    \item If $A\in \mathcal{Q}$ and $B$ is finite then $A\setminus B\in \mathcal{Q}.$
\end{enumerate}
\end{definition}

We note a basic property of ultrafilters.
\begin{lemma}
If $A\cup B\in \mathcal{Q}$, $\mathcal{Q}\subset P(X)$ an ultrafilter, then either $A\in \mathcal{Q}$ or $B\in \mathcal{Q}.$
\end{lemma}

\begin{proof}
Suppose neither $A$ nor $B$ is in $\mathcal{Q}.$ Then $X\setminus A\in \mathcal{Q}$ and $X\setminus B\in \mathcal{Q},$ by property 4 of definition \ref{defintion nonprinc uf}. Hence $X\setminus (A\cup B) = (X\setminus A)\cap (X\setminus B)\in \mathcal{Q}$, hence $A\cup B\not\in \mathcal{Q}.$
\end{proof}

The following lemma will be useful in proving the existence of ultrafilters over any infinite set\cite[section 1.2]{garcia2012filters}.
\begin{lemma}
The union of a chain of filters is again a filter.
\end{lemma}
\begin{proof}
Let $\{\mathcal{F}_i\}$ be a chain of filters partially ordered by set inclusion, so that $\mathcal{F}_i\subset\mathcal{F}_j$ if $j>i.$ Let $A,B\in \cup_i\mathcal{F}_i$. Then $A\in \mathcal{F}_i$, $B\in\mathcal{F}_j$ for some $i,j.$ Hence $A\cup B\in \mathcal{F}_{\max\{i,j\}}$, since the chain is partially ordered. Hence $A\cup B\in \mathcal{F}_i$ or $\mathcal{F}_j.$\par

Now to show the union has property 2, let $A,B\in \cup_i \mathcal{F}_i.$ As before, $A\in \mathcal{F}_i$, $B\in\mathcal{F}_j$ for some $i,j,$ so $A, B$ are both in $\mathcal{F}_{\max\{i,j\}}.$ Since this is a filter $A\cap B\in \mathcal{F}_{\max\{i,j\}}\subset \cup_i\mathcal{F}_i.$
\end{proof}

Note that the set of all cofinite sets is a non-principal, proper filter on $X$. This filter is called the Fr\'echet Filter, and any filter containing this filter is called a free filter. We can prove the existence of a non-principal free ultrafilter on any infinite set using Zorn's Lemma \cite[page 205]{hrbacek1984introduction}\cite{garcia2012filters}. The proof of the existence of an ultrafilter over $X$, an infinite set, follows from a dichotomy property guaranteeing that a maximal proper filter on $X$ is an ultrafilter, as shown below.

\begin{lemma}
There exists a non-principal ultrafilter on $X$, $X$ an infinite set.
\end{lemma}
\begin{proof}
Let $F$ be the set of free filters on $X$. Note $F$ is not empty as the Fr\'echet filter is in $F.$ We can partially order $F$ by set inclusion. Since the union of any chain of free filters is again a free filter, each chain is bounded above by the union of all the filters in the chain. Hence by Zorn's lemma there exists a maximal filter in $F.$ Call it $\mathcal{Q}.$ We show $\mathcal{Q}$ is a non-principle ultrafilter. To do so we need show only property 4 of definition \ref{defintion nonprinc uf} above, as $\mathcal{Q}$ is in $F$ and hence a filter.\par

Note for all $A\subset X,$ $A\cup (X\setminus A)\in \mathcal{Q}.$ Suppose for contradiction neither $A$ nor $X\setminus A$ is in $\mathcal{Q}.$ If this is the case, we claim $T = \{U\in P(X)| A\cup U\in \mathcal{Q}\}$ is a free filter properly containing $\mathcal{Q}$.
\begin{itemize}
    \item $T$ is a filter.
    \begin{itemize}
        \item Monotonicity: Let $\alpha\in T,$ and $\alpha\subset\beta.$ Then $A\cup\alpha\in \mathcal{Q}$. But $A\cup\alpha\subset A\cup\beta$, so by the monotonicity of $\mathcal{Q}$ we have $A\cup\beta\in\mathcal{Q}$. Hence $\beta\in T$.
        \item Closure under intersection: If $\alpha,\beta\in T$ then $A\cup \alpha\in \mathcal{Q}$, and similarly for $\beta.$ Then $A\cup(\alpha \cap \beta)=(A\cup \alpha)\cap(A\cup \beta)\in \mathcal{Q}$ since $\mathcal{Q}$ is a filter. Hence $\alpha\cap\beta\in T.$
    \end{itemize}
    \item $\mathcal{Q}\subseteq T:$ Let $\alpha\in \mathcal{Q}.$ Then by monontonicity $A\cup \alpha\in \mathcal{Q}$, so $\alpha\in T.$ 
    \item $T$ properly contains $\mathcal{Q}:$ This is clear from the fact that $X = A\cup (X\setminus A),$ hence $X\setminus A\in T$ but $X\setminus A\not\in\mathcal{Q}.$
\end{itemize}
Hence, supposing $\mathcal{Q}$ does not have the dichotomy property we come to a contradiction, namely, that $Q$ is not a maximal filter. Therefore $Q$ has the dichotomy property, property 4 of definition \ref{defintion nonprinc uf} above.
\end{proof}

\subsection{Using Ultrafilters to Define Limits}
Now that we have an ultrafilter we can define limits of sequences and functions. We first give an intuitive, algorithmic definition of the limit of a sequence, as in \cite{tao_2007}, then show this is equivalent to a more convenient definition.\par

\begin{algorithm}[Computing the Non-standard Limits of Bounded Sequences]\label{algorithm defn of nonstd limit}

Let $\{a_n\}\subset \R$ be a bounded sequence, $a_n\in (x_1,z_1)$ for all $n.$ We can view $\{a_n\}$ as a function $a:\N\to \R.$ Let $\mathcal{Q}$ be a non-principal ultrafilter on $\N.$ Compute the $\mathcal{Q}$-limit of $\{a_n\}$ by choosing $y_1\in (x_1,z_1)$ and determining whether $\inv{a}((x_1,y_1))=U_1$ or $\inv{a}([y_1,z_1))=V_1$ is in $\mathcal{Q}.$ Note that $U_1\cap V_1=\varnothing$, so only one of $U_1$ or $V_1$ is in $\mathcal{Q}$. (Moreover, $U_1\cup V_1=\N$). Say $U_1\in \mathcal{Q}.$ Repeat with a $y_i\in (U_{i-1})$, obtaining a sequence of nested intervals $\{U_i\}$ with width tending to $0$ for which $\inv{a}(U_i)\in \mathcal{Q}.$ The limit of this sequence of intervals is a point, which is the non-standard limit or $\mathcal{Q}$-limit denoted by $\nstdlm{\mathcal{Q}}(a_n).$ If this non-standard limit equals $l$, we may write $\{a_n\}\underset{\mathcal{Q}}{\rightarrow}l.$
\end{algorithm}

For this algorithm to be useful, it will have to yield a unique limit regardless of the choice of intervals used to compute it. Fortunately it does. Consider sequences of nested intervals $\{U_i\}\to l$ and $\{V_i\}\to l'.$ If these 2 sequences yield the same limit using the above algorithm, then computing the limit using $U_i\cap V_i$ shows that $l=l'.$

It is important to note that the limit of the sequence is not shift-invariant, as shifting the index may affect which inverse image is in $\mathcal{Q}$ \cite{tao_2007}. Therefore the non-standard limit is not, in general, characterized sequentially for functions, and we will have to introduce other tools to define the limits of functions. The above lemma motivates the introduction of a more convenient characterization of the sequential $\mathcal{Q}$-limit of a sequence.\footnote{See, for example, \cite[page 590]{kapovich1995asymptotic}. Compare with \cite[definition 2.4]{gromov1996geometric}.}

\begin{theorem}\label{inverse image characterization of nonstandard limits}
Let $\{a_n\}$ be a bounded sequence, and $\mathcal{Q}$ be a nonprincipal ultrafilter on $\N.$ Then $\nstdlm{\mathcal{Q}}(a_n) = l$ if and only if, for all open intervals $U$ containing $l$, $\inv{a}(U)\in \mathcal{Q}.$
\end{theorem}

\begin{proof}
Suppose first all intervals $U$ containing $l$ are such that $\inv{a}(U)\in \mathcal{Q}.$ Let $\{U_i\}\to l$ be any sequence of nested intervals converging to $l.$ Then $\nstdlm{\mathcal{Q}}(a_n)=l$ by algorithm \ref{algorithm defn of nonstd limit}.\par

Now suppose $\nstdlm{\mathcal{Q}}(a_n)=l.$ Let $U$ be any interval containing $l.$ Then either $\inv{a}(U)$ or $\inv{a}(\mathbb{R}\setminus U)$ is in $\mathcal{Q}.$ If the latter then for no subinterval of $U$ is the inverse image under $a$ in $\mathcal{Q}.$ Therefore the $\mathcal{Q}$-limit cannot be any point of $\mathbb{R}\setminus U.$
\end{proof}

Note that the dichotomy property of an ultrafilter guarantees this limit always exists and is unique, moreover it obeys the algebra homomorphism laws for limits, in particular $\nstdlm{\mathcal{Q}}(a_n+b_n) = \nstdlm{\mathcal{Q}}(a_n)+\nstdlm{\mathcal{Q}}(b_n)$.

\begin{lemma}\label{algebra homomorphism}
If $\nstdlm{\mathcal{Q}}(a_n)= l$, $\nstdlm{\mathcal{Q}}(b_n)= l'$ then $\nstdlm{\mathcal{Q}} (a_n+b_n)=l+l'.$
\end{lemma}
\begin{proof}
Note since for every open neighborhood $U$ of $l+l'$ there exist neighborhoods $V,V$' of $l$ and $l'$ respectively such that $V+V'\subset U$ it suffices to show that $(a+b)^{-1}(V+V')\in Q$. Let $V,V'$ be as above, and let $m\in a^{-1}(V)\cap b^{-1}(V').$ Then there exist $a_m,b_m$ such that $a_m\in V$ and $b_m\in V'.$ Hence $a_m+b_m\in V+V'$, moreover, $m\in (a+b)^{-1}(V+V').$ Therefore $a^{-1}(V)\cap b^{-1}(V')\subseteq (a+b)^{-1}(V+V').$ Hence $(a+b)^{-1}(V+V')\in Q$ by the monotonicity and closure under intersection properties of ultrafilters. The lemma follows from theorem \ref{inverse image characterization of nonstandard limits}.
\end{proof}

To define the nonstandard limit of a function it is necessary to specify not only an ultrafilter $\mathcal{Q}$ on $\N$ but also a sequence, say $\{\varepsilon_n\}$, tending to $0.$ This inspires the following definition for the nonstandard limit of a function.

\begin{definition}
Let $\{a_n\}$ be a bounded sequence in the domain of $f:\R^m\to\R^n$. Let $f_n = f(a_n)$ be bounded. Then $\nstdlm{\mathcal{Q},a_n}f(x) = \nstdlm{\mathcal{Q}}f_n.$
\end{definition}

Recall, from lemma \ref{dyadic cubes lemma}, that for $2^{-(k+1)}<\varepsilon\leq 2^{-k}$ and $A$ a bounded subset of $\R^n$ that $S_A(2^{-k})$ was defined as the minimal number of dyadic cubes of side length $2^{-k}$ covering $A.$ In what follows let $S_A(\varepsilon)=S_A(2^{-k})$ with $k$ as above.

\begin{definition}\label{nonstd mink dim}
Let $\mathcal{Q}$ be an ultrafilter on $\N,$ $\{\varepsilon_n\}$ be a sequence of positive reals tending to $0.$ Let $X\subset\R^n$ be bounded. Then the \textit{nonstandard Minkowski dimension} of $X$ relative to $\mathcal{Q},\{\varepsilon_n\}$ is
\begin{align}
    \dim\limits_{\mathcal{Q},\varepsilon_n}(X) = \lim\limits_{\mathcal{Q},\varepsilon_n} \frac{\log S(\varepsilon_n)}{\log\frac{1}{\varepsilon_n}}.
\end{align}
\end{definition}

From the algebraic homomorphism properties of the nonstandard limits we can establish the main result of this paper, namely the product summability of the nonstandard Minkowski dimension. Compare definition \ref{nonstd mink dim} with \cite[theorem 4.3]{potgieter}.

\begin{theorem}\label{nonstdmink product summable}
Let $X\subset \R^m, Y\subset \R^n,$ $\mathcal{Q},\{\varepsilon_n\}$ as in definition \ref{nonstd mink dim}. Then
\begin{align*}
    \dim\limits_{\mathcal{Q},\varepsilon_n} (X\times Y) = \dim\limits_{\mathcal{Q},\varepsilon_n} (X) + \dim\limits_{\mathcal{Q},\varepsilon_n} (Y).
\end{align*}
\end{theorem}
\begin{proof}
Note $\frac{\log(S_{X\times Y}(\varepsilon))}{\log(\frac{1}{\varepsilon})} = \frac{\log(S_X(\varepsilon))}{\log(\frac{1}{\varepsilon})}+\frac{\log(S_Y(\varepsilon))}{\log(\frac{1}{\varepsilon})}$ by lemma \ref{lemma s atimesb}. The result follows from lemma \ref{algebra homomorphism}.
\end{proof}

We now have a dimension, $\dim\limits_{\mathcal{Q},\varepsilon_n},$ which is product summable and exists for every bounded set, since the nonstandard limit always exists. Moreover, the nonstandard Minkowski dimension agrees with the Minkowski dimension when it exists.

\section{Example of Two Sets Which Are Not Product Summable}\label{sets not product summable section}
In \textit{The Geometry of Fractal Sets} Falconer discusses an example of two subsets $A,B$ of $[0,1]$, each of which have Hausdorff dimension $\dimh(A)=\dimh(B) = 0$, but which have a Cartesian product of dimension at least $1$ \cite[page 73]{falconer1986geometry}. Falconer constructs such sets by using a sequence of subsets for which the Hausdorff dimension tends to $0$. For more on the Hausdorff dimension, consult \cite{falconer1986geometry}. Here we present a generalization which is an interesting case of when the nonstandard Minkowski dimension, the Hausdorff dimension, and standard upper Minkowski dimension fundamentally differ, because of their respective product-summability.\par

Let $A_1,B_1,A_2,B_2,...$ be a partition of $\N,$ so $A_1\bigcup B_1\bigcup A_2\bigcup B_2\bigcup ... = \{1,2,3,...\}$, with sets $A_i,B_j$ which are the intersections of connected sets in $\R$ with $\N$.\par

We consider expansions in base $d$. Define the sets $A,B$ as the subsets of the unit interval whose elements have $i$th digits missing $B_i$ in base $d,$ so

\begin{align}
    A &= \{x=\sum\limits_{i=1}\limits^\infty a_id^{-i},\textrm{ }a_i = 0\textrm{ if }i\in B_k\textrm{ for some }k\}.\end{align} Define $B$ similarly,
\begin{align}
B &= \{x=\sum\limits_{i=1}\limits^\infty b_id^{-i},\textrm{ }b_i = 0\textrm{ if }i\in A_k\textrm{ for some }k\}.
\end{align}

We use closed intervals of length $d^{-(m+1)}<\varepsilon \leq d^{-m}$ to cover $A.$ Note that $N_A(\varepsilon)$ is the number of choices of digits up to $a_m,$ because intervals of length $\varepsilon$ are insensitive to variations in size less than or equal to $d^{-(m+1)}.$ Denote the number of different combinations in $A$ and $B$ by $f_A$ and $f_B$ respectively, so $f_A = \abs{\left(\bigcup_{i=1}^\infty A_i\right) \bigcap\{1,...,m\}}$ and $f_B = \abs{\left(\bigcup_{i=1}^\infty B_i\right) \bigcap\{1,...,m\}}.$ Note $f_A+f_B=m$. Then $N_A(\varepsilon)$ and $N_B(\varepsilon)$ are, for $d^{-(m+1)}<\varepsilon\leq d^{-m}$,

\begin{align}
    N_A(\varepsilon) = d^{f_A} \hspace{1.5em}\textrm{and}\hspace{1.5em} N_B(\varepsilon)=d^{f_B}.
    \end{align}
If $\varepsilon=d^{-m}$ then
\begin{align*}
   \frac{\log N_A(\varepsilon)}{\log (\frac{1}{\varepsilon})} =\frac{f_A}{m}\hspace{1.5em}\textrm{and}\hspace{1.5em}
    \frac{\log N_B(\varepsilon)}{\log (\frac{1}{\varepsilon})} =\frac{f_B}{m}.
\end{align*}

We can construct the sets $A$ and $B$ so the function $\frac{\log N(\varepsilon)}{\log(1/\varepsilon)}=\frac{f}{m}$ oscillates between $0$ and $1$ as $\varepsilon$ tends to $0$, by controlling the size of each set $A_i$ or $B_i$ (as in the example below). When $f_A\approx m$ the upper Minkowski dimension of $A$ will be observed to be near $1$ while the dimension of $B$ will be observed to be near zero. Hence the limit infimum is $0$ and the limit supremum is $1$ for both, by construction and because $f_A+f_B=m$. Thus the limit does not exist for these sets.

Because the upper Minkowski dimension is a limit supremum, $\dimup(A)=\dimup(B) = 1.$ Similarly, the Hausdorff dimension is an infimum, so $\dimh(A)=\dimh(B) = 0.$ 

As an example, let $A_1=\{1\}$, $B_1 = \{2\}$, and for $n\geq1$ define $\abs{A_{n+1}} =n\abs{\cup_{i=1}^n(A_i\cup B_i)}$ and $\abs{B_{n+1}} = n\abs{\cup_{i=1}^n(A_i\cup B_i)\cup A_{n+1}}$. Then $A_2=\{3,4\},$ $B_2 = \{5,6,7,8\},$ and the sets are completely determined and each next block of digits is a multiple of the lengths of all the previous blocks combined. Let $m=\max A_{n+1}.$ Then
\begin{align*}
    f_A&\geq m-\frac{m}{n+1}=\frac{n}{n+1}m\\
    f_B&\leq \frac{m}{n+1}.
\end{align*}
Then $\dimup(A)$ is observed to be near $1$ when measured just as $A_{n+1}$ is complete, while $\dimup(B)$ when measured at this scale would be near $0.$ Note $\dim_{\mathcal{Q}}(A)$ and $\dim_{\mathcal{Q}}(B)$ exist and are between $0$ and $1$. Note there is some freedom to choose $\dim_A$ by choosing $\varepsilon_n.$\par

\subsection{The Dimensions of $A\times B$}
Here we demonstrate by counterexample that the Hausdorff and upper Minkowski dimensions are not product-summable. Before doing so, we note the following inequality between the upper Minkowski dimension and the Hausdorff dimension. This results, essentially, from the fact that the Hausdorff dimension is defined in terms of an infimum over all possible covers while the upper Minkowski dimension is computed using sets of a fixed size \cite[equation 1.2.3]{bishop2017fractals}.

\begin{theorem}\label{minkgeqhaus}
Let $A\subset \R^n$ be bounded. Then

\begin{equation}\label{dimhleqdimup}
    \dimup(A)\geq\dimh(A).
\end{equation}
\end{theorem}

If we define $A+B = \{x+y|x\in A,\textrm{ } y\in B\},$ then $\dimm(A+B) = 1$ because, for all $y\in [0,1],$ there exists an $a\in A$ and a $b\in B$ such that $a+b = y$; therefore $A+B = [0,1]$ which is a 1-dimensional set. And indeed,
\begin{align}
\frac{\log N_A(\varepsilon)}{\log(1/\varepsilon)}+\frac{\log N_B(\varepsilon)}{\log(1/\varepsilon)} = \frac{|A_1|+|A_2|+...+|A_k\cap\{1,...,m\}|}{m}+\frac{|B_1|+...+|B_{k-1}|}{m} = 1.
\end{align}
We start with the Hausdorff dimension $\dimh(A\times B).$ Note that projections are distance-nonincreasing, and so do not increase the Hausdorff outer measure. Let $\proj((x_0,y_0))$ be the orthogonal projection of $(x_0,y_0)\in\R^2$ onto the line $y=x$, so $\proj((x_0,y_0)) = \left(\frac{1}{2}(x+y),\frac{1}{2}(x+y)\right)$. Note that $\proj(A\times B)$ has the same dimension as the interval $[0,1]$. Hence $\dimh(A\times B)\geq\dimh(\proj(A\times B))\geq 1,$ so $\dimh(A\times B)>\dimh(A)+\dimh(B)$ and the Hausdorff dimension is not product-summable.\par

We can also show the upper Minkowski dimension is not product-summable. Let $d^{-(m+1)}<\varepsilon\leq d^{-m}$, then $N_{A\times B}(\varepsilon)\leq N_A(\varepsilon)N_B(\varepsilon)=d^{f_A+f_B}.$ Since $f_A+f_B=m$, we have $N_{A\times B}(\varepsilon)\leq d^{-m}= \varepsilon^{-1}.$ Hence $\dimup(A\times B)= 1,$ and $\dimup(A)=\dimup(B)=1$ as well.\par

It is here the nonstandard Minkowski dimension shows its utility. Let $\mathcal{Q}$ be an ultrafilter on $\N$ and $\varepsilon_n$ be a sequence of positive numbers tending to 0. Here $N_{A\times B}(\varepsilon) = S_{A\times B}(\varepsilon)$. Since $N_A(\varepsilon)N_B(\varepsilon) = d^m,$ we have $
    \dim\limits_{\mathcal{Q},\varepsilon_n}(A\times B) = \lim\limits_{\mathcal{Q},\varepsilon_n}\frac{\log(S_{A\times B}(\varepsilon_n))}{\log(1/\varepsilon_n)}
    =1.$ We conclude $\dim\limits_{\mathcal{Q},\varepsilon_n}(A\times B) = \dim\limits_{\mathcal{Q},\varepsilon_n}(A)+\dim\limits_{\mathcal{Q},\varepsilon_n}(B).$

\bibliographystyle{unsrt}  
\bibliography{main}  

\end{document}